\documentclass[12pt]{amsart}
\usepackage{amscd}
\usepackage{amsmath}
\usepackage{amssymb}
\usepackage{units}
\usepackage[all]{xy}

\usepackage{algorithm}
\usepackage{algorithmic}

\def\frk{\frak}               % font for "Fraktur"

\def\pp{{\frk p}}

\def\mm{{\frk m}}

\def\Phi{{\frk n}}
\def\Phi{{\frk N}}
\def\opn#1#2{\def#1{\operatorname{#2}}} % to make operators
\opn\chara{char} \opn\length{\ell} \opn\pd{pd} \opn\rk{rk}
\opn\projdim{proj\,dim} \opn\injdim{inj\,dim} \opn\rank{rank}
\opn\depth{depth} \opn\sdepth{sdepth} \opn\fdepth{fdepth}
\opn\grade{grade} \opn\height{height} \opn\embdim{emb\,dim}
\opn\codim{codim}  \opn\min{min} \opn\max{max}

\opn\Tr{Tr} \opn\bigrank{big\,rank}
\opn\superheight{superheight}\opn\lcm{lcm}
\opn\trdeg{tr\,deg}%\emph{
\opn\reg{reg} \opn\lreg{lreg} \opn\ini{in} \opn\lpd{lpd}
\opn\size{size}
\opn\div{div} \opn\Div{Div} \opn\cl{cl} \opn\Cl{Cl}
\opn\Spec{Spec} \opn\Supp{Supp} \opn\supp{supp} \opn\Sing{Sing}
\opn\Ass{Ass} \opn\Min{Min}
\opn\Ann{Ann} \opn\Rad{Rad} \opn\Soc{Soc}
\opn\Im{Im} \opn\Ker{Ker} \opn\Coker{Coker} \opn\Am{Am}
\opn\Hom{Hom} \opn\Tor{Tor} \opn\Ext{Ext} \opn\End{End}
\opn\Aut{Aut} \opn\id{id}  \opn\deg{deg}

\opn\nat{nat}
\opn\pff{pf}%   \pf exists already
\opn\Pf{Pf} \opn\GL{GL} \opn\SL{SL} \opn\mod{mod} \opn\ord{ord}
\opn\Gin{Gin} \opn\Hilb{Hilb}
\opn\aff{aff} \opn\con{conv} \opn\relint{relint} \opn\st{st}
\opn\lk{lk} \opn\cn{cn} \opn\core{core} \opn\vol{vol}
\opn\link{link} \opn\star{star}
\opn\gr{gr}

\def\pot#1#2{#1[\kern-0.28ex[#2]\kern-0.28ex]}

\opn\dirlim{\underrightarrow{\lim}}
\opn\inivlim{\underleftarrow{\lim}}
\let\to=\rightarrow

\def\Implies{\ifmmode\Longrightarrow \else
        \unskip${}\Longrightarrow{}$\ignorespaces\fi}
\def\implies{\ifmmode\Rightarrow \else
        \unskip${}\Rightarrow{}$\ignorespaces\fi}
\def\iff{\ifmmode\Longleftrightarrow \else
        \unskip${}\Longleftrightarrow{}$\ignorespaces\fi}

\let\:=\colon
\newtheorem{Theorem}{Theorem}[]
\newtheorem{Lemma}[Theorem]{Lemma}
\newtheorem{Corollary}[Theorem]{Corollary}
\newtheorem{Proposition}[Theorem]{Proposition}
\newtheorem{Remark}[Theorem]{Remark}

\let\epsilon\varepsilon
\let\phi=\varphi
\let\kappa=\varkappa
\def\qed{\ifhmode\textqed\fi
      \ifmmode\ifinner\quad\qedsymbol\else\dispqed\fi\fi}
\def\textqed{\unskip\nobreak\penalty50
       \hskip2em\hbox{}\nobreak\hfil\qedsymbol
       \parfillskip=0pt \finalhyphendemerits=0}
\def\dispqed{\rlap{\qquad\qedsymbol}}

\opn\dis{dis}
\def\pnt{{\raise0.5mm\hbox{\large\bf.}}}

\opn\Lex{Lex}

\begin{document}
\title{  A uniform General Neron Desingularization in dimension one}

\author{Asma Khalid, Gerhard Pfister,  and Dorin Popescu }

\address{Asma Khalid, Abdus Salam School of Mathematical Sciences,GC University, Lahore, Pakistan}
\email{asmakhalid768@gmail.com}

\address{Gerhard Pfister,  Department of Mathematics, University of Kaiserslautern, Erwin-Schr\"odinger-Str., 67663 Kaiserslautern, Germany}
\email{pfister@mathematik.uni-kl.de}

\address{Dorin Popescu, Simion Stoilow Institute of Mathematics of the Romanian Academy, Research unit 5,
University of Bucharest, P.O.Box 1-764, Bucharest 014700, Romania}
\email{dorin.popescu@imar.ro}

\begin{abstract} We give a uniform General Neron Desingularization for one \linebreak dimensional local rings with respect to morphisms which coincide modulo a high power of the maximal ideal. The result has interesting applications in the case of Cohen-Macaulay rings.

\end{abstract}

\maketitle

\vskip 0.5 cm

\section*{Introduction}

Solving equations is an old problem in mathematics. When we deal with \linebreak polynomial equations over Noetherian complete local rings then a special case is  the {\em smooth} case, that is a system where we may apply the implicit function theorem. If possible one tries   to replace an arbitrary polynomial system of equations by a smooth one. This is done in the  theorem below,  which was used for example in  Artin's Approximation Theory \cite{A} (see also \cite{P}, \cite{P2}).

Let us recall some definitions. A ring morphism $u:A\to A'$ of Noetherian rings has  {\em regular fibers} if for all prime ideals $\pp\in \Spec A$ the ring $A'/\pp A'$ is a regular  ring, i.e. its localizations are regular local rings.
It has {\em geometrically regular fibers}  if for all prime ideals $\pp\in \Spec A$ and all finite field extensions $K$ of the fraction field of $A/\pp$ the ring  $K\otimes_{A/\pp} A'/\pp A'$ is regular.

A flat morphism of Noetherian rings $u$ is {\em regular} if its fibers are geometrically regular. If $u$ is regular of finite type then $u$ is called {\em smooth}. A localization of a smooth algebra is called {\em essentially smooth}. A Henselian Noetherian local ring $A$ is {\em excellent} if the completion map $A\to \hat A$ is regular.

\begin{Theorem} (General Neron Desingularization,  Andr\'e \cite{An}, Popescu \cite{P}, \cite{P1}, Swan \cite{S}, Spivakovski \cite{Sp})\label{gnd}  Let $u:A\to A'$ be a  regular morphism of Noetherian rings and $B$ an  $A$-algebra of finite type. Then  any $A$-morphism $v:B\to A'$ factors through a smooth $A$-algebra $C$, that is $v$ is a composite $A$-morphism $B\to C\to A'$.
\end{Theorem}

The $A$-algebra $C$ is called a General Neron Desingularization (shortly GND) and in the case when $(A',\mm')$ is local of dimension one it   is constructive (see \cite{AP}, \cite{PP}). The construction depends on formal power series defining $v$. We have to replace a given system of polynomials by a smooth one using only a polynomial approximation of $v$.

The purpose of this paper is to give such construction which depends only
on a solution $\bar y$  in $A'$ of $f$ modulo $\mm'^m$ for some $m>>0$ (see Theorem \ref{arc}). Moreover, the constructed GND works for each solutions $y$ of $f$ in $A'$ lifting $\bar y$. We call such GND {\em uniform}.

As a consequence (see Remark \ref{r1}) we get \cite[Theorem 14]{PP} (see also \cite[Theorem 20]{AP}), that is there exists
 a  linear map similar as in Greenberg's case (see \cite{Gr}).

An extension of \cite[Theorem 10]{P2} is  Theorem \ref{arc1}, which shows that if $A,A'$ are Cohen-Macaulay rings and $v:B\to A'/\mm'^{3k}$ an $A$-morphism, then the set of all \linebreak $A-$ morphisms $B\to A'$ equivalent with $v$ modulo $\mm'^{3k}$,  is in bijection with  $\mm^{3k}A'^s$ for some $s\in {\bf N}$.  Theorem \ref{arc1} has also a variant in the case when $(A,\mm)$ is Artinian and $(A',\mm')$ is a Noetherian local ring of dimension one, which says that for an $A$-algebra $B$  of finite type and $v:B\to A'/\mm'^{3k}$ an $A$-morphism, the set of all \linebreak $A$-morphisms $B\to A'$ equivalent with $v$ modulo $\mm'^{3k}$, is in bijection with  $\mm'^{3k}A'^s$ for some $s\in {\bf N}$ (see Theorem \ref{arc2}).

We owe thanks to the Referees who showed us several misprints and had some useful comments on our presentation.

\section{Uniform  General Neron Desingularization}

We begin recalling some definitions. Let $A$ be a Noetherian   ring,
 $B=A[Y]/I$, $Y=(Y_1,\ldots,Y_n)$. If $f=(f_1,\ldots,f_r)$, $r\leq n$ is a system of polynomials in $I$ then we can define the ideal $\Delta_f$ generated by all $r\times r$-minors of the Jacobian matrix $(\partial f_i/\partial Y_j)$.   After Elkik \cite{El} let $H_{B/A}$ be the radical of the ideal $\sum_f ((f):I)\Delta_fB$, where the sum is taken over all systems of polynomials $f$ from $I$ with $r\leq n$.
Then $B_{\pp}$, $\pp\in \Spec B$, is essentially smooth over $A$ if and only if $\pp\not \supset H_{B/A}$ by the Jacobian criterion for smoothness.
Thus  $H_{B/A}$ measures the non smooth locus of $B$ over $A$. $B$ is {\em standard smooth} over $A$ if there exists  $f$ in $I$ as above such that $B= ((f):I)\Delta_fB$.

Now let us fix some notations and assumptions for the next theorem.
 Let $u:A\to A'$ be a flat morphism of Noetherian local rings of dimension $1$ and $B,I$ be as above. Suppose that the maximal ideal $\mm$ of $A$ generates the maximal ideal of $A'$ and the completions of $A,A'$ are isomorphic. Moreover suppose that $A'$ is Henselian and $u$ is a regular morphism.
Let $(0)=\cap_{\pp\in \Ass A} Q_{\pp}$ be a reduced primary decomposition
of $(0)$ in $A$, where $Q_{\pp}$ is a primary ideal with $\sqrt{Q_{\pp}}=\pp$. Let  $e\in {\bf N}$  be such that ${\pp}^e\subset Q_{\pp}$ for all $\pp\in \Ass A$.

\begin{Theorem}\label{arc} With the notations $B,I,e$ fixed above, let  $f=(f_1,\ldots,f_r)$, $r\leq n$ be a system of polynomials in $I$, $M$ an $r\times r$-minor  of the Jacobian matrix $(\partial f_i/\partial Y_j)$ and  $k,c\in \bf N$. Suppose that there exist $N\in ((f):I)$ and an $A$-morphism $v:B\to A'/\mm^{(2e+1)k+c}A'\cong A/\mm^{(2e+1)k+c}$, given by $Y\to y'\in A^n$, such that the ideal generated by $d=(MN)(y')$ in
$A/\mm^{(2e+1)k}$ contains $\mm^k/\mm^{(2e+1)k}$. Then there exists a $B$-algebra $C$ which is standard smooth over $A$ with the following  properties:

\begin{enumerate}
\item \label{p1}
 Every $A$-morphism $v':B\to A'$ with $v'\equiv v \ \mbox{modulo}\ d^{2e+1}A'$ (that is $v'(Y)\equiv v(Y) \ \mbox{modulo}\
 d^{2e+1}A'$) factors through $C$.

\item \label{p2}
 Every $A$-morphism $v':B\to A'$ with $v'\equiv v \ \mbox{modulo}\ \mm^{(2e+1)k}A'$ factors through $C$.

 \item \label{p3}
 There exist an $A$-morphism $w:C\to A'$ which makes the following diagram commutative
 $$
  \begin{xy}\xymatrix{B \ar[d] \ar[r] & C  \ar[r]^{w} & A'\ar[d]\\
  A/\mm^{(2e+1)k+c} \ar[r] &  A/\mm^c \ar[r]   & A'/\mm^cA' }
  \end{xy}
  $$

 \end{enumerate}
\end{Theorem}

\begin{proof}
 We follow  the proof of \cite[Theorem 10]{P2} and
 \cite[Theorem 2]{PP}.
    Set $P=NM$.  As $A$ is not Artinian  we  see that $d$ is not nilpotent. We have
 $(0:_Ad^e)=(0:_Ad^{e+1})$ because of the choice of $e$.
Note that
 $\mm^k\subset dA+\mm^{(2e+1)k+c}\subset dA+\mm^{(2e+1)((2e+1)k+c)+c}\subset \ldots $. Thus $\mm^k\subset dA$, in particular $\mm^{(2e+1)k+c}\subset dA$.

 Note  that  $dA/\mm^{(2e+1)k+c}$ is the ideal corresponding to $v(NM)A'/\mm^{(2e+1)k+c}A'$ by the isomorphism  $A/\mm^{(2e+1)k+c}\cong A'/\mm^{(2e+1)k+c}A'$.

  We may assume that $M=\det((\partial f_i/\partial Y_j)_{i,j\in [r]})$, where $[r]=\{1,\ldots,r\}$.  Let  $H$ be the $n\times n$-matrix obtained by adding down to $(\partial f/\partial Y)$ as a border the block $(0|\mbox{Id}_{n-r})$ (we assume as above that $M$ is given by the first $r$ columns of the Jacobian matrix). Let $G'$ be the adjoint matrix of $H$ and $G=NG'$. We have
$$GH=HG=NM \mbox{Id}_n=P\mbox{Id}_n$$
and so
$$d\mbox{Id}_n=P(y')\mbox{Id}_n=G(y')H(y').$$

Let ${\bar A}=A/(d^{2e+1})$, ${\bar B}={\bar A}\otimes_AB$, ${\bar A}'=A'/(d^{2e+1}A')$. Since $\mm^k\subset dA$ we see that $v$ defines modulo $d^{2e+1}$ a map
 ${\bar v}:{\bar B}\to {\bar A'}$.

 Let
 \begin{equation}\label{h}
   h=Y-y'-d^eG(y')T
 \end{equation}
 
 where  $T=(T_1,\ldots,T_n)$ are new variables.  Thus
 \begin{equation}\label{mod_h}
   Y_l-y'_l\equiv d^eG_l(y')T\ \mbox{modulo}\ h_l
 \end{equation}
 
for $l\in [n]$, where $G_l(y')$ means the row $l$ of the matrix $G(y')$. Since
$$f(Y)-f(y')\equiv \sum_j\partial f/\partial Y_j(y') (Y_j-y'_j)$$
modulo higher order terms in $Y_j-y'_j$, by Taylor's formula we see that  we have
$$f(Y)-f(y')\equiv  \sum_jd^e\partial f/\partial Y_j(y') G_j(y')T+d^{2e}Q=$$
$$d^e P(y')T+d^{2e}Q=d^{e+1}(T+d^{e-1}Q)\
\mbox{modulo}\ h,$$
 where $Q\in T^2 A[T]^r$ and modulo $h$ means modulo $(h)A[Y,T]^r$. This is because we replace everywhere $Y_j-y'_j$ by $ d^eG_j(y')T$ modulo $h_j$ and use that
  $(\partial f/\partial Y)G=(P\mbox{Id}_r|0)$. Since $f(y')\in \mm^{(2e+1)k+c}\subset d^{2e+1}\mm^c$ we have $f(y')=d^{e+1}a$ for some $a\in d^e\mm^cA^r$. Set
$g_i=a_i+T_i+d^{e-1}Q_i$, $i\in [r]$ and  $E=A[Y,T]/(I,g,h)$. Clearly, it holds that $d^{e+1}g\subset (f,h)$ and $(f)\subset (h,g)$ .

Let $s$ be the first  $r\times r$-minor of the Jacobian matrix of $g$.  Thus $s=\det(\mbox{Id}_r+d^{e-1}(\partial Q_i/\partial T_j)_{i,j\in [r]})\in 1+(T)$ and  $U=(A[T]/(g))_{s}$ is standard smooth. In \linebreak particular, $U$ is flat over $A$ and so $(0:_Ud^e)=(0:_Ud^{e+1})$.
We claim that $E_{ss'}\cong U_{s'}$ for some $s'\in 1+(d,T)$. It will be enough to show that $I\subset (h,g)A[Y,T]_{ss'}$ \linebreak because then we get $E_{ss'}\cong A[Y,T]_{ss'}/(g,h)\cong U_{s'}$. We have $PI\subset (f)\subset (h,g)$ and so $P(y'+d^eL)I\subset (h,g)$ where $L=G(y')T$. But $P(y'+d^eL)$ has the form $ds'$ for some $s'\in 1+(d^{e-1}T)\subset 1+(d^e,T)$. It follows that $s'I\subset ((h,g):d)A[Y,T]_{s}$, that is $s'IA[Y,T]_s/(h,g)\subset (0:_{A[Y,T]_s/(h,g)}\ d)\cong (0:_Ud)$. On the other hand, $I\equiv I(y')\equiv 0$  modulo $(d^e,h)$
because $I(y')\subset \mm^{(2e+1)k}\subset (d^{2e+1})$ since $\mm^k\subset (d)$ and so $s'IA[Y,T]_s/(h,g)=0$
since $(0:_Ud^e)=(0:_Ud^{e+1})$ implies $(0:_Ud)\cap d^eU=0$. Thus $I\subset  (h,g)A[Y,T]_{ss'}$ and it follows that
 $C=E_{ss'}$ is a $B$-algebra  smooth over $A$. Note that by construction $C$ is a smooth $A$-algebra of standard form.

It remains to see that an arbitrary $A$-morphism $v':B\to A'$ with
 $v'\equiv v$\ modulo $ \mm^{(2e+1)k}$  factors through $C$. We have $v'(Y)\equiv y' \
\mbox{modulo}\ \mm^{(2e+1)k}$ and so $v'(Y)\equiv y'\
\mbox{modulo}\ d^{2e+1}$.  Thus there exists $\epsilon\in d^eA'^n$ such that $v'(Y)-y'=d^{e+1}\epsilon$.
 Then $t:=H(y')\epsilon\in d^eA'^n$
satisfies
$$G(y')t=P(y')\epsilon=d\epsilon$$
 and so
 $$v'(Y)- y'=d^eG(y')t,$$
that is $h(v'(Y),t)=0$.
Note that $d^{e+1}g(t)\in (h(v'(Y),t), f(v'(Y)))=(0)$
and $g(t)\in d^e A'^r$. It follows that $g(t)\in (0:_{A'^r}d^{e+1})\cap d^eA'^r=0 $ because $(0:_{A'}d^e)=(0:_{A'}d^{e+1})$.
 Thus $v'$ factors through $E$, that  is $v'$ is a composite map $B\to E\xrightarrow{\alpha} A'$, where $\alpha$ is a $B$-morphism given by $T\to t$.
 As $\alpha(s)\equiv 1$ modulo $d$ and $\alpha(s')\equiv 1$ modulo $(d,t)$, $t\in dA'$ we see that $\alpha(s),\alpha(s')$ are invertible because  $A'$ is local  and so $\alpha$ (thus $v'$) factors through the standard smooth $A$-algebra $C$. Thus (1)  holds. Then (2) holds from (1), because in this case we get $v'\equiv v $ modulo $d^{2e+1}$ since $\mm^k\subset dA$.

 Now  for (3) take the map ${\hat w}:C\cong U_{s'}\to A'/\mm^cA'$ given by $(Y,T)\to (y',0)$. Then the composite map $B\to C\xrightarrow{\hat w} A'/\mm^c A'$ is given by   $Y\to v(Y)\ \mbox{modulo}\ \mm^c$. Since $C$ is standard smooth, we may lift $\hat w$ to an $A$-morphism $w:C\to A'$ by the Implicit Function Theorem. Clearly, $w$ makes the above diagram commutative because $v(Y)\ \mbox{modulo}\ \mm^c$ corresponds to $w(Y)\ \mbox{modulo}\ \mm^cA'$ by the isomorphism $A/\mm^c\cong A'/\mm^cA'$.
\hfill\ \end{proof}

\begin{Remark} \label{r} {\em The number $c$ in the above theorem is necessary only in (3) and for (1) (and so (2)) we could take $c=0$ from the beginning. Since (1) and (3) have mainly the same proof, we did not want to repeat the arguments and so we unify (1), (3) in the above theorem, the disadvantage being to have $c$ not useful in (1).}
 \end{Remark}

\begin{Remark}\label{r1} {\em The third statement of the above theorem gives a variant of
\cite[Theorem 14]{PP} in the idea of Greenberg \cite{Gr}. More precisely, let $(A,\mm) $ be an excellent Henselian local ring, $A'={\hat A}$ and $B,\ I$ be as above. Suppose that $\hat y$ is a solution of $I$ in $A/\mm^{(2e+1)k}\cong {\hat A}/\mm^{(2e+1)k}{\hat A}$ such that there exist a system of polynomials $f=(f_1,\ldots,f_r)$ in $I$, an $r\times r$ minor $M$ of the Jacobian matrix $(\partial f/\partial Y)$  and $N\in ((f):I)$ such that $(MN)({\hat y}){\hat A}/ \mm^{(2e+1)k}{\hat A}\supset \mm^k{\hat A}/\mm^{(2e+1)k}{\hat A}$. As a consequence of Theorem
\ref{arc} (3), every solution $\tilde y$ of $I$ in $A/\mm^{(2e+1)k+c}$, $c\in {\bf N}$ lifting $\hat y$ can be lifted to a solution $y$ of $I$ in $A$. When $A$ is a DVR then by Greenberg \cite{Gr} there exists a linear map $\nu:{\bf N}\to {\bf N}$ such that every solution $\tilde y$ of $I$ in $A/\mm^{\nu(c)}$, $c\in {\bf N}$ can be lifted to a solution $y$ of $I$ in $A$. By \cite{PP0} and \cite{P} there exist such maps $\nu$ (not necessarily linear) for every excellent Henselian local ring $A$. 
Our above consequence gives similarly a linear map $\nu$ of a special form $c\to (2e+1)k+c$ but this is not the map given by Greenberg because here $k$ (and so $\nu$) depends on $\tilde y$.}
\end{Remark}

\begin{Remark}\label{r2}  {\em If $A$ is reduced then $e=1$ and $v:B\to A'/\mm^{3k+c}A'$. In fact, the proof goes similarly with  $v:B\to A'/\mm^{3k+c}A'$ (in (1) with $v:B\to A'/d^3A'$) if $A$ is Cohen-Macaulay, because for the proof in this case it is enough to see that $d$ is regular which follows from $\mm^k\subset (d)$.}
\end{Remark}

The following proposition follows  in particular from \cite[Theorem 2]{El}
but it is also a consequence of Theorem \ref{arc}.

\begin{Proposition} (Elkik)\label{e} Let $(A,\mm)$ be a Noetherian Henselian local ring of dimension one and $B=A[Y]/I$, $Y=(Y_1,\ldots,Y_n) $ an $A$-algbra of finite type. Then for every $k\in \bf N$ there exist two integers
$m_0,p\in \bf N$ such that if $y'\in A^n$ satisfies $\mm^k\subset H_{B/A}(y')$ and $I(y')\equiv 0$ modulo $\mm^m$ for some    $m>m_0$ then there exists $y\in A^n$ such that $I(y)=0$ and $y\equiv y'$ modulo $\mm^{m-p}$.
\end{Proposition}

\begin{proof} Suppose that $A'=A$. In the notations of Theorem \ref{arc} given $k$ set $m_0=p=(2e+1)k$ and  suppose that $y'\in A^n$ satisfies  $\mm^k\subset H_{B/A}(y')$ and $I(y')\equiv 0$ modulo $\mm^m$ for some    $m>m_0$. Let $v:B\to A/\mm^m$ be given by $Y\to y'$. Set $c=m-p$. By  Theorem \ref{arc} there exists a standard smooth
$A$-algebra $C$ and a map $w:C\to A$ which makes the  diagram commutative from (3) of Theorem \ref{arc}. Let $y$ be the image of $Y$ by the composite map $B\to C\xrightarrow{w} A$. Then  we have $I(y)=0$ and $y\equiv y'$ modulo $\mm^c=\mm^{m-p}$.
\hfill\ \end{proof}

\begin{Remark} \label{r3} {\em Proposition \ref{e} still holds when $A$ is a Noetherian Henselian local ring of arbitrary dimension (see \cite[Theorem 2]{El}).  Thus for fixed $k\in \bf N$ and $B=A[Y]/I$ as above there exist $m_0,p\in \bf N$ such that the linear map $c\to p+c$ behaves almost like the Artin function of $B$ when $y'\in A^n$ satisfies $H_{B/A}(y')\supset \mm^k$ and $c\geq m_0-p$. As in Remark \ref{r1}, $k$ depends on $y'$ and we cannot say that the Artin function of $B$ is linear.

Let $f$ be a  system of $r$ polynomials from $I$ as above, $M$ an $r\times r$-minor of $(\partial f/\partial Y)$ and $N\in ((f):I)$. The above sentences give the idea that there exists a $\rho\in \bf N$ such that if there exists an $A$-morphism $v:B\to A/\mm^\rho$ with  $v(MN)\supset \mm^k/\mm^{\rho}$ for some $k\in \bf N$ then there exists a $B$-algebra $C$, which is smooth over $A$, such that any $A$-morphism $v':B\to \hat A$ with $v'\equiv v $ modulo $\mm^{\rho}\hat A$ factors through $C$, $\hat A$ being the completion of $A$. In other words, we believe that Theorem \ref{arc} holds in arbitrary dimension if the smooth locus of $B$ is big with respect to $v$, that is if $v(H_{B/A}){\hat A}$ is $\mm {\hat A}$ primary. Note that \cite[Remark 4.7]{R} (see also \cite[Remark 16]{PP})
almost disagrees with our belief.}
\end{Remark}

\section{The Cohen-Macaulay case.}

Here we present some consequences of Theorem \ref{arc} when $A$ is Cohen-Macaulay and so $e=1$ by Remark \ref{r2}.
\begin{Corollary}\label{c1} With the assumptions and notations  of
 Theorem \ref{arc},  suppose that $A$ is Cohen-Macaulay, and let $\rho:B\to C$ be the structural algebra map. Then $\rho$ induces bijections $\rho^*$ given by $\rho^*(w)=w\circ \rho$, between
 \begin{enumerate}
   \item $\{w\in \Hom_A(C,A'):w\circ \rho\equiv v \ \mbox{modulo}\ d^{3}A'\}$ and
 $\{v'\in \Hom_A(B,A'):v'\equiv v \ \mbox{modulo}\ d^{3}A'\},$
   \item $\{w\in \Hom_A(C,A'):w\circ \rho\equiv v \ \mbox{modulo}\ \mm^{3k}A'\}$ and
 $\{v'\in \Hom_A(B,A'):v'\equiv v \ \mbox{modulo}\ \mm^{3k}A'\}.$
 \end{enumerate}
 \end{Corollary}

 \begin{proof}
 \begin{enumerate}
 \item By Theorem \ref{arc}, (1), also by Remark \ref{r2}, $\rho^*$ is surjective. Let $w,w'\in \Hom_A(C,A')$ be such that $w\circ \rho=w'\circ \rho\equiv v\ \mbox{modulo}\ d^{3}A'$. In particular $w|_B=w'|_B$. In the notations which are introduced in the proof of Theorem \ref{arc}, and from equation \ref{mod_h} we have $H(y')(Y-y')\equiv d^{2}T \mod h$ by using $G(y')H(y')=d\mbox{Id}_n.$
Applying $w$ and $w'$ on the above congruence and subtracting we get
 $H(y')(w(Y)-w'(Y))= d^2(w(T)-w'(T))$, that is $d^{2}(w(T)-w'(T))=0$ because $w(Y)=w'(Y)$ and so $w(T)=w'(T)$ because $d$ is regular on $A'$ since $d$ is regular in $A$ (see Remark \ref{r2}) and $u$ is flat. It follows that $w=w'$.

   \item Apply Theorem \ref{arc}, (2) and  Remark \ref{r2} for surjectivity, and injectivity goes as above.

 \end{enumerate}

\hfill\ \end{proof}

 By  construction, $C$ has the form $(A[T]/(g))_{ss'}$ in the notations of Theorem \ref{arc}, where $s=\det (\partial g_i/\partial T_j)_{i,j\in [r]} $. Let  ${\tilde w}:C\to A'$  be  the unique $A$-morphism such that ${\tilde w}\circ \rho\equiv  v \ \mbox{modulo}\ d^3A'$ in the notations of Corollary \ref{c1}. We have  ${\tilde w}(ss')\not \in \mm A'$.

\begin{Lemma}\label{lem} There  exist canonical bijections
\begin{enumerate}
  \item $$ A'^{n-r}\to  \{w'\in \Hom_A(C,A'):w'\equiv {\tilde w} \ \mbox{modulo}\ d^{3}A'\}.$$
  \item $$ \mm^{3k} A'^{n-r}\to  \{w'\in \Hom_A(C,A'):w'\equiv {\tilde w} \ \mbox{modulo}\ \mm^{3k}A'\}.$$
\end{enumerate}
\end{Lemma}

\begin{proof} Let $w'\in \Hom_A(C,A')$ be with $w'\equiv {\tilde w}\ \mbox{modulo}\ d^{3}A'$.
We have
$$g(w'(T))\equiv g({\tilde w}(T))=0\ \mbox{modulo}\ d^{3}A',$$
$$(ss')(w'(T))\equiv (ss')({\tilde w}(T))\not \equiv 0\ \mbox{modulo}\ \mm.$$
Set $V=(T_1,\ldots,T_r)$, $Z=(T_{r+1},\ldots,T_n)$.
Thus $g(V,w'(Z))=0$ has a unique solution  (namely $w'(V)$) in ${\tilde w}(Z)+d^{3}A'$ by the Implicit Function Theorem.
It follows that the restriction
$$\{w'\in \Hom_A(C,A'):w'\equiv {\tilde w} \ \mbox{modulo}\ d^{3}A'\}\to $$
$$\{w''\in \Hom_A(A[Z],A'):w''\equiv {\tilde w}|_{A[Z]} \ \mbox{modulo}\ d^{3}A'\}$$
is bijective.

On the other hand, the map
$$\{w''\in \Hom_A(A[Z],A'):w''\equiv {\tilde w}|_{A[Z]} \ \mbox{modulo}\ d^{3}A'\}\to {\tilde w}(Z)+ d^{3}A'^{n-r}$$
given by $w''\to w''(Z)$ is a bijection too. It follows that there exists a canonical bijection between
$\{w'\in \Hom_A(C,A'):w'\equiv {\tilde w} \ \mbox{modulo}\ d^{3}A'\}$ and  ${\tilde w}(Z)+ d^{3}A'^{n-r}$, the last one being in bijection with $d^3A'^{n-r}$, that is with $A'^{n-r}$, $d$ being regular in $A'$. The proof of (2) goes similarly.
\hfill\ \end{proof}

The following theorem extends \cite[Theorem 15]{P2}.
\begin{Theorem} \label{arc1} With the assumptions and notations of Corollary \ref{c1} there exist \linebreak canonical bijections

\begin{enumerate}
  \item $$A'^{n-r}\to \{v'\in \Hom_A(B,A'):v'\equiv v \ \mbox{modulo}\ d^{3}A' \}.$$
  \item $$\mm^{3k}A'^{n-r}\to \{v'\in \Hom_A(B,A'):v'\equiv v \ \mbox{modulo}\ \mm^{3k}A' \}.$$
\end{enumerate}

 \end{Theorem}
For the proof apply Corollary \ref{c1} and the above lemma.

Let $(\Lambda,\pp)$ be a local Artinian ring, $(A',\mm')$ be a local Noetherian ring of dimension $1$ and $u:\Lambda \to A'$ a regular morphism.
Suppose that $A'$ is Henselian and has the same residue field with $\Lambda$.
 Let $F$ be an  $\Lambda$-algebra of finite type, let us say $F=\Lambda[U]/J$, $U=(U_1,\ldots,U_n)$. A constructive GND of $F$ with respect to an $A$-morphism $F\to A'$ follows easily from the classical N\'eron Desingularization \cite{N} as it is shown in \cite{KK}. Note that $A'/{\pp}A'$ is a DVR because it is a one dimensional  regular local ring, $u$ being regular. Let $x\in A$ be a  lifting of a local parameter of $A/\pp$. Clearly, $x$ is a regular element in $A$ and $\Lambda,A'$ are Cohen-Macaulay rings.
Set $A=\Lambda[X]_{(\pp,X)}$, $B=A\otimes_\Lambda F$, and let $u':A \to A'$ be the regular morphism extending $u$ by  $X\to x$. Let
$f=(f_1,\ldots,f_r)$, $r\leq n$ be a system of polynomials from $J$ and $M$ an $r\times r$-minor  of the Jacobian matrix $(\partial f_i/\partial U_j)$,  $g:F\to A'/\mm'^{3k}$  a $\Lambda$-morphism, and  $N\in ((f):J)$.
Suppose that $g(NM) A'/\mm'^{3k} \supset \mm'^k/\mm'^{3k}$.

The following theorem extends \cite[Corollary 16]{P2}.
  \begin{Theorem}\label{arc2} In the assumptions and notations from above, the set ${\mathcal G}=\{g'\in \Hom_\Lambda(F,A'):g'\equiv g \ \mbox{modulo}\ \mm'^{3k}\}$ is in bijection with $\mm'^{3k} A'^{n-r}$.
\end{Theorem}

For the proof apply Theorem \ref{arc1} to the morphism $A\otimes_\Lambda F \to A'$ induced by $g$ and $u'$ and note that $A\otimes_\Lambda$ induces a bijection $\Hom_\Lambda (F,A')\to \Hom_A(B,A')$ by adjunction.

\section{Algorithm N\'eron Desingularization}

Input:
$ A=k[x]_{(x)}/J,$ $J=(h_1,\ldots,h_q),\ h_i\in k[x], x=(x_1,\ldots,x_t), \ \dim(A)=1,$ k is a field.
$B=A[Y]/I, I=(f_1,\ldots, f_l), f_i\in k[x,Y], Y=(Y_1,\ldots,Y_n)$, integers $k$,$r$, $y' \in {k[x]}^n$, $N \in (f_1,\ldots,f_r):I$.

Output: A Neron Desingularization induced by $y'$ or the message ``$y'$, $N$, $(f_1,\ldots,f_r)$ are not well chosen''.

\begin{enumerate}
  \item Compute $M=\det((\partial f_i/\partial Y_j)_{i,j\in [r]})$, $P:=NM$ and $d:=P(y')$
  \item $f:= (f_1,\ldots,f_r)$
  \item Compute $e$ such that $(0:_Ad^e)=(0:_Ad^{e+1})$
  \item If $I(y')\nsubseteq (x)^{(2e+1)k} +J$ or $(x)^k \nsubseteq (d)+J$, return ``$y'$, $N$, $(f_1,\ldots,f_r)$ are not well chosen''
  \item Complete $(\partial f_i/\partial Y_j)_{i \leq r}$ by $(0| (\mbox{Id}_{n-r}))$ to obtain a square matrix $H$
  \item Compute $G'$ the adjoint matrix of $H$ and $G:=NG'$
  \item $h=Y-y'-d^e G(y')T,\  T=(T_1,\ldots, T_n)$
  \item Write $f(Y)-f(y')=  \sum_jd^e\partial f/\partial Y_j(y') G_j(y')T+d^{2e}Q$
  \item Write $f(y')=d^{e+1}a$
  \item For $i=1$ to $r$, $g_i=a_i+T_i+d^{e-1}Q_i$
  \item $E:=A[Y,T]/(I,g,h)$
  \item Compute $s$ the $r\times r$ minor defined by the first $r$ columns of $(\partial g/\partial T) $
  \item Write $P(y'+ d^eG(y')T)=ds'$
  \item return $E_{ss'}$.
\end{enumerate}

{\bf Acknowledgements} The first author  gratefully acknowledges the support from the ASSMS GC. University Lahore, for arranging her visit to Bucharest, Romania and she is also grateful to the Simion Stoilow Institute of the Mathematics of the Romanian Academy for inviting her.

\end{document}